\documentclass[a4paper,11pt]{amsart}
\usepackage{latexsym}
\usepackage{amscd}
\usepackage{amssymb}
\usepackage{amsfonts}
\usepackage{mathrsfs}
\usepackage[all]{xy}
\usepackage{amsmath, pb-diagram}
\usepackage{enumerate}

\DeclareMathAlphabet{\eusm}{OT1}{eusm}{m}{n}
\newtheorem{theor}{Theorem}[section]
\newtheorem{prop}[theor]{Proposition}
\newtheorem{defi}[theor]{Definition}
\newtheorem{cor}[theor]{Corollary}
\newtheorem{lem}[theor]{Lemma}
\newtheorem{exam}[theor]{Example}
\newtheorem{remark}[theor]{Remark}

\begin{document}

\title[Relatively divisible and relatively flat objects]
{Relatively divisible and relatively flat objects \\ in exact categories}



\author[S. Crivei]{Septimiu Crivei}
\address {Faculty of Mathematics and Computer Science, Babe\c{s}-Bolyai University, 400084 Cluj-Napoca, Romania}
\email{crivei@math.ubbcluj.ro}

\author[D. Kesk\.{i}n T\"ut\"unc\"u]{Derya Kesk\.{i}n T\"ut\"unc\"u}
\address {Department of Mathematics, Hacettepe University, 06800 Ankara, Turkey} \email{keskin@hacettepe.edu.tr}

\begin{abstract} We introduce and study relatively divisible and relatively flat objects in exact categories in the
sense of Quillen. For every relative cotorsion pair $(\mathcal{A},\mathcal{B})$ in an exact category $\mathcal{C}$, 
$\mathcal{A}$ coincides with the class of relatively flat objects of $\mathcal{C}$ for some relative projectively generated
exact structure, while $\mathcal{B}$ coincides with the class of relatively divisible objects of $\mathcal{C}$ for some relative 
injectively generated exact structure. 
We exhibit Galois connections between relative cotorsion pairs, relative projectively generated exact
structures and relative injectively generated exact structures in additive categories. We establish closure properties
and characterizations in terms of approximation theory. 
\end{abstract}

\date{October 27, 2018}

\maketitle

\section{Introduction}

Many important characterizations of rings are of homological algebra nature. Some of the best known examples are the
following ones: a ring $R$ is semisimple if and only if every short exact sequence of right (or left) $R$-modules
splits; a ring $R$ is von Neumann regular if and only if every short exact sequence of right (or left) $R$-modules is
pure; a ring $R$ is right pure-semisimple if and only if every pure short exact sequence of right $R$-modules splits.
These examples exhibit three classes of short exact sequences in a module category, namely split short exact
sequences, pure short exact sequences and all short exact sequences. It has been noticed years ago that such classes of
short exact sequences share some properties, that can be considered as axioms for a more general notion, called
\emph{exact structure} (Keller \cite{Keller}, Quillen \cite{Q}) or \emph{proper class} (Buchsbaum \cite{Buch}) on an
additive category. It turns out that this context of exact structures on additive categories is a suitable one for
developing homological algebra in categories more general than module categories or abelian categories. First, it
encompasses non-abelian categories, which appear naturally in algebraic geometry (Rosenberg \cite{Rosenberg}),
functional analysis (Frerick, Sieg \cite{FS}) etc. Secondly, they allow the use of relative homological algebra in
abelian categories (Enochs, Jenda \cite{EJ}, Garkusha \cite{G04}). 

As in the above examples, one may have many exact structures on the same category, and it is useful to find
relationships between them as well as characterizations of objects of the category in terms of such exact structures
defined on it. We mention a couple of other examples, out of many having the root in abelian group theory. Recall that a
short exact sequence of right $R$-modules $0\to A\to B\to C\to 0$ is called \emph{neat} if every simple right $R$-module
is projective with respect to it, and \emph{closed} if $A$ is isomorphic to a complement (closed) submodule of $B$
(Renault \cite{Renault}). Then every closed short exact sequence of right $R$-modules is neat, while the converse holds
if and only if $R$ is a right $C$-ring in the sense of \cite{Renault}. Also, a short exact sequence of right $R$-modules
is called \emph{finitely split} if every finitely generated right $R$-module is projective with respect to it (Azumaya
\cite{Az87}). Then every finitely split short exact sequence of right $R$-modules is pure, while the converse holds if
and only if $R$ is a right Noetherian ring \cite[Proposition~6]{Az87}. Further examples of classes of short exact
sequences giving exact structures on module categories include \emph{coneat} (Fuchs \cite{Fuc}), \emph{$s$-pure} (I.
Crivei, S. Crivei \cite{CC}) or \emph{supplement} (Mermut \cite{Engin}) short exact sequences.

To each pair of exact structures $\mathcal{D}$ and $\mathcal{E}$ on an additive category $\mathcal{C}$ there are some
associated objects. An object $X$ of $\mathcal{C}$ is called \emph{$\mathcal{D}$-$\mathcal{E}$-divisible} if every short
exact sequence from $\mathcal{D}$ (called \emph{$\mathcal{D}$-conflation}) starting with $X$ belongs to $\mathcal{E}$,
and \emph{$\mathcal{D}$-$\mathcal{E}$-flat} if every $\mathcal{D}$-conflation ending with $X$ belongs to $\mathcal{E}$.
They are inspired by Sklyarenko's modules of flat type relative to a proper class \cite{Skly}, and have been
reconsidered by Preisser Monta\~no \cite[15.4,15.11]{Montano} in abelian categories. One of the two exact structures
$\mathcal{D}$ and $\mathcal{E}$ is often that given by the class of all short exact sequences in the category, but we
point out that this class fails to form an exact structure unless the underlying category is not quasi-abelian (see
Example \ref{ex1}). Examples of $\mathcal{D}$-$\mathcal{E}$-divisible objects in suitable categories include: injective,
pure-injective, absolutely pure (i.e., every short exact sequence starting with it is pure \cite{Maddox,Megibben}), finitely
injective (i.e., every short exact sequence starting with it is finitely split \cite{Az87}), finitely pure-injective
(i.e., every pure short exact sequence starting with it is finitely split \cite{Az87}), absolutely neat (i.e., every
short exact sequence starting with it is neat \cite{Sep}), absolutely coneat (i.e., every short exact sequence starting
with it is coneat \cite{BD14}), absolutely $s$-pure (i.e., every short exact sequence starting with it is $s$-pure
\cite{CC}), weak injective (i.e., every short exact sequence starting with it is closed \cite{Z06}) objects etc. Each of
the above situations generates a corresponding notion of $\mathcal{D}$-$\mathcal{E}$-flat object, namely projective,
pure-projective, flat, finitely projective \cite{Az87}, finitely pure-projective \cite{Az87}, neat-flat \cite{BD16},
coneat-flat \cite{BD14}, max-flat \cite{Xiang} and weak flat \cite{Z13} object respectively. 

The objective of the present paper is to give a unified treatment of relative divisibility and relative flatness in the
natural setting of exact categories. Our main motivating examples come from module theory, but we also present some
application to finitely accessible additive categories, which may not be abelian in general. We do not intend to
give an exhaustive study of the concepts of relative divisibility and relative flatness, but rather to set the stage
with some relevant properties and illustrations, and to leave the interested reader to deduce herself/himself the needed
consequences. Most of the results will have two parts, one related to relative divisibility and one related to relative
flatness, out of which we will only prove one, the other following in a dual manner.

Our paper is organized as follows. We recall in Section 2 the axioms of an exact structure on an additive
category in the sense of Quillen, as refined by Keller. We also present some needed concepts, such as
relative projectively and relative injectively generated exact structures, relative cotorsion pairs and approximation
theory by relative covers and relative envelopes. 

In Section 3 we show how the classes $Div(\mathcal{D}\textrm{-}\mathcal{E})$ and
$Flat(\mathcal{D}\textrm{-}\mathcal{E})$ of relatively divisible and relatively flat objects respectively allow some
natural constructions of two Galois connections, generalizing the original work by Salce \cite{Salce}. More precisely,
let $\mathcal{C}$ be an additive category with an exact structure $\mathcal{D}$, and denote by $DCot(\mathcal{C})$ the
class of $\mathcal{D}$-cotorsion pairs in $\mathcal{C}$. Also, denote by $DPEx(\mathcal{C})$ and $DIEx(\mathcal{C})$
the classes of $\mathcal{D}$-projectively and $\mathcal{D}$-injectively generated exact structures on
$\mathcal{C}$ respectively. Then we prove that there are a monotone Galois connection between $DPEx(\mathcal{C})$ and
$DCot(\mathcal{C})$, and an antitone Galois connection between $DIEx(\mathcal{C})$ and $DCot(\mathcal{C})$. Moreover, we
extend and dualize the notion of injectively generated Xu exact structure, and we show that the above Galois connections
restrict to bijective correspondences between $\mathcal{D}$-cotorsion pairs in $\mathcal{C}$, $\mathcal{D}$-projectively
generated Xu exact structures on $\mathcal{C}$, and $\mathcal{D}$-injectively generated Xu exact structures on
$\mathcal{C}$.

Section 4 contains the main properties of relatively divisible and relatively flat objects in additive categories. Let
$\mathcal{C}$ be an additive category, and let $\mathcal{D}$ and $\mathcal{E}$ be exact structures on $\mathcal{C}$. For
a class $\mathcal{A}$ of objects of $\mathcal{C}$, we define an object $X$ of $\mathcal{C}$ to be
\emph{$\mathcal{D}$-$\mathcal{E}$-$\mathcal{A}$-divisible} if every $\mathcal{D}$-conflation $X\rightarrowtail
A\twoheadrightarrow B$ is an $\mathcal{E}$-conflation, and \emph{$\mathcal{D}$-$\mathcal{E}$-$\mathcal{A}$-flat} in a
dual manner. We show several closure properties of the classes of $\mathcal{D}$-$\mathcal{E}$(-$\mathcal{A}$)-divisible
and $\mathcal{D}$-$\mathcal{E}$(-$\mathcal{A}$)-flat objects of $\mathcal{C}$. Their level of generality allows many
useful consequences, such as the equivalence of the $\mathcal{D}$-resolving property of $\mathcal{A}$ and the
$\mathcal{D}$-coresolving property of $\mathcal{B}$ for every $\mathcal{D}$-cotorsion pair $(\mathcal{A},\mathcal{B})$
in $\mathcal{C}$, provided $\mathcal{C}$ has enough $\mathcal{D}$-injectives and enough $\mathcal{D}$-projectives. We
also characterize $\mathcal{D}$-$\mathcal{E}$(-$\mathcal{A}$)-divisible
and $\mathcal{D}$-$\mathcal{E}$(-$\mathcal{A}$)-flat objects of $\mathcal{C}$ in terms of the existence of certain
conflations involving them. 

Applications of our results are included in a separate paper \cite{CK2}.

\section{Exact categories}

We shall use the following concept of exact category given by Quillen \cite{Q} (also see Buchsbaum \cite{Buch}), as
simplified by Keller \cite{Keller}.

\begin{defi} \rm An \emph{exact category} is an additive category $\mathcal{C}$ endowed with a
distinguished class $\mathcal{E}$ of short exact sequences (i.e. kernel-cokernel pairs) satisfying the axioms $[E0]$,
$[E1]$, $[E2]$ and $[E2^{\rm op}]$ below. Then $\mathcal{E}$ is called an \emph{exact structure} on $\mathcal{C}$. The
short exact sequences in $\mathcal{E}$ are called \emph{$\mathcal{E}$-conflations} (or simply \emph{conflations}),
while kernels and cokernels appearing in such exact sequences are called \emph{inflations} (denoted by
$\rightarrowtail$) and \emph{deflations} (denoted by $\twoheadrightarrow$) respectively. \vskip2mm

$[E0]$ The identity morphism $1_0:0\to 0$ is a deflation.

$[E1]$ The composition of two deflations is again a deflation.

$[E2]$ The pullback of a deflation along an arbitrary morphism exists and is again a deflation.

$[E2^{\rm op}]$ The pushout of an inflation along an arbitrary morphism exists and is again an inflation.
\end{defi}

\begin{remark} \label{r:max} \rm The duals of the axioms $[E0],[E1]$ on inflations as well as both sides of
Quillen's ``obscure axiom'' hold in any exact category \cite{Keller}. The version of the latter for
inflations states that if $i,p$ are morphisms such that $i$ has a cokernel and $pi$ is an inflation,
then $i$ is an inflation. 
\end{remark}

\begin{exam} \label{ex1}

\rm (1) Any additive category has a unique minimal exact structure, whose conflations are the split short
exact sequences (e.g., see \cite{Buhler}).

(2) An additive category is called \emph{quasi-abelian} if it is pre-abelian (i.e., it has kernels
and cokernels), any pushout of a kernel along an arbitrary morphism is a kernel, and any pullback of a cokernel along an
arbitrary morphism is a cokernel. Any quasi-abelian category, and in particular any abelian category, has a unique
maximal exact structure, whose conflations are the kernel-cokernel pairs \cite{Rump01}. 

(3) An additive category is called \emph{weakly idempotent complete} if every split monomorphism has a cokernel, or
equivalently, every split epimorphism has a kernel (e.g., see \cite{Buhler}). A kernel (cokernel) in an additive
category is called \emph{semi-stable} if its pushout (pullback) along any morphism exists and is again a kernel
(cokernel) (see \cite[Definition~2.4]{C12}). Any weakly idempotent complete additive category has a unique maximal exact
structure, whose conflations are the kernel-cokernel pairs consisting of a semi-stable kernel and a semi-stable cokernel
\cite[Theorem~3.5]{C12}. Note that every additive category has an idempotent-splitting (Karoubian) completion (see
\cite[p.~75]{Karoubi}), which is weakly idempotent complete. 
\end{exam}

\begin{defi} \rm Let $\mathcal{C}$ be an additive category with an exact structure $\mathcal{E}$. An object $X$ of
$\mathcal{C}$ is called \emph{$\mathcal{E}$-injective} if every $\mathcal{E}$-conflation starting with $X$ splits.
Equivalently, $X$ is $\mathcal{E}$-injective if and only if ${\rm Ext}^1_{\mathcal{E}}(C,X)=0$ for every object $C$ of
$\mathcal{C}$. The category $\mathcal{C}$ is said to have enough \emph{$\mathcal{E}$-injectives} if for every object $X$
of $\mathcal{C}$ there exists an $\mathcal{E}$-conflation $X\rightarrowtail Y\twoheadrightarrow Z$ with $Y$
$\mathcal{E}$-injective. 

Dually, one defines \emph{$\mathcal{E}$-projective} objects and categories having enough $\mathcal{E}$-projectives.
\end{defi}

The reader is referred to \cite[Section~11]{Buhler} for further properties of injective and projective objects in exact
categories. 

The following propositions can be easily deduced (also see \cite[Section~4]{Montano} and \cite{St67}). They provide
ways to reflect exact structures via functors.

\begin{prop} Let $F:\mathcal{C}\to \mathcal{C}'$ be a covariant (respectively contravariant) additive functor between 
additive categories $\mathcal{C}$ and $\mathcal{C}'$ with exact structures $\mathcal{D}$ and $\mathcal{D}'$
respectively. Then the class $\mathcal{E}_F$ of $\mathcal{D}$-conflations $X\to Y\to Z$ in $\mathcal{C}$ such that
$F(X)\rightarrowtail F(Y)\twoheadrightarrow F(Z)$ $(\textrm{respectively } F(Z)\rightarrowtail F(Y)\twoheadrightarrow
F(X))$ is a $\mathcal{D}'$-conflation in $\mathcal{C}'$ defines an exact structure on $\mathcal{C}$.
\end{prop}

\begin{prop} \label{p:projinjgen} Let $\mathcal{C}$ be an additive category with an exact structure $\mathcal{D}$ and
${\rm Ab}$ the category of abelian groups. For any object $M$ of $\mathcal{C}$, consider the additive functors $H={\rm
Hom}_{\mathcal{C}}(M,-)$, $G={\rm Hom}_{\mathcal{C}}(-,M):\mathcal{C}\to {\rm Ab}$. For a class $\mathcal{M}$ of objects
of $\mathcal{C}$, denote \[\mathcal{E}_H^{\mathcal{M}}=\bigcap \{\mathcal{E}_H \mid M\in \mathcal{M}\} \textrm{ and }
\mathcal{E}_G^{\mathcal{M}}=\bigcap \{\mathcal{E}_G \mid M\in \mathcal{M}\}.\] Then $\mathcal{E}_H^{\mathcal{M}}$ and
$\mathcal{E}_G^{\mathcal{M}}$ are exact structures on $\mathcal{C}$, which are called the exact structures {\it
$\mathcal{D}$-projectively generated by $\mathcal{M}$} and {\it $\mathcal{D}$-injectively generated by $\mathcal{M}$}
respectively, and are also denoted by $\pi^{-1}_{\mathcal{D}}(\mathcal{M})$ and $\iota^{-1}_{\mathcal{D}}(\mathcal{M})$
respectively.
\end{prop}

When all kernel-cokernel pairs form an exact structure $\mathcal{D}$ on an additive category $\mathcal{C}$, we omit
$\mathcal{D}$ in Proposition \ref{p:projinjgen} and we talk about \emph{projectively} and \emph{injectively}
generated exact structures. 

We recall some categorical results on pushouts and pullbacks for easier reference. 

\begin{lem} \cite[Proposition~2.12]{Buhler} \label{l:POPB} Let $\mathcal{C}$ be an additive category with an exact
structure.
\begin{enumerate}
\item Consider the following diagram in $\mathcal{C}$, where the rows are conflations: 
$$\SelectTips{cm}{}
\xymatrix{
X \ar@{>->}[r]^i \ar[d]_f & Y \ar@{->>}[r]^d \ar[d]^g & Z \ar@{=}[d] \\
X' \ar@{>->}[r]_{i'} & Y' \ar@{->>}[r]_{d'} & Z 
}$$
Then the square $XYY'X'$ is both a pushout and a pullback. 
\item Consider the following diagram in $\mathcal{C}$, where the rows are conflations: 
$$\SelectTips{cm}{}
\xymatrix{
X \ar@{>->}[r]^i \ar@{=}[d] & Y \ar@{->>}[r]^d \ar[d]^g & Z \ar[d]^h \\
X \ar@{>->}[r]_{i'} & Y' \ar@{->>}[r]_{d'} & Z' 
}$$
Then the square $YZZ'Y'$ is both a pushout and a pullback. 
\end{enumerate}
\end{lem}

\begin{lem} \cite[Lemma~5.1]{Kelly} \label{l:rec} Consider the following diagram in an arbitrary category
$\mathcal{C}$: 
$$\SelectTips{cm}{}
\xymatrix{
X \ar[r]^i \ar[d]_f & Y \ar[r]^d \ar[d]^g & Z \ar[d]^h \\
X' \ar[r]_{i'} & Y' \ar[r]_{d'} & Z 
}$$
\begin{enumerate} 
\item 
\begin{enumerate}[(i)]
\item If the two squares are pushouts, then the rectangle is a pushout.
\item Assume that the left square is a pushout. Then the exterior rectangle is a pushout if and only if the right square
is a pushout.
\end{enumerate}
\item 
\begin{enumerate}[(i)]
\item If the two squares are pullbacks, then the rectangle is a pullback.
\item Assume that the right square is a pullback. Then the exterior rectangle is a pullback if and only if the left
square is a pullback.
\end{enumerate}
\end{enumerate}
\end{lem}

Let $\mathcal{C}$ be an additive category with an exact structure $\mathcal{D}$. One may construct the Yoneda bifunctor
${\rm Ext}^1_{\mathcal{D}}(-,-)$. For objects $X$ and $Z$ of $\mathcal{C}$, one has ${\rm Ext}^1_{\mathcal{D}}(Z,X)=0$
if and only if every conflation $X\rightarrowtail Y\twoheadrightarrow Z$ in $\mathcal{C}$ splits, that is, it is
isomorphic to the canonical split exact sequence $X\rightarrowtail X\oplus Z\twoheadrightarrow Z$. 

For a class $\mathcal{A}$ of objects of $\mathcal{C}$, we denote \[\mathcal{A}^{\perp_{\mathcal{D}}}=\{X\in
\mathcal{C}\mid {\rm Ext}^1_{\mathcal{D}}(A,X)=0 \textrm{ for every } A\in \mathcal{A}\},\]
\[{}^{\perp_{\mathcal{D}}}\mathcal{A}=\{X\in \mathcal{C}\mid {\rm Ext}^1_{\mathcal{D}}(X,A)=0 \textrm{ for every } A\in
\mathcal{A}\}.\]

The notion of cotorsion pair generalizes from abelian categories to exact categories in an obvious manner. 

\begin{defi} \cite{G,Hovey} \rm Let $\mathcal{C}$ be an additive category with an exact structure $\mathcal{D}$. 
A \emph{$\mathcal{D}$-cotorsion pair} in $\mathcal{C}$ is a pair of classes $(\mathcal{A},\mathcal{B})$ of objects in
$\mathcal{C}$ such that $\mathcal{A}^{\perp_{\mathcal{D}}}=\mathcal{B}$ and
${}^{\perp_{\mathcal{D}}}\mathcal{B}=\mathcal{A}$.

Let $\mathcal{A}$ be a class of objects of $\mathcal{C}$. The pair 
$({}^{\perp_{\mathcal{D}}}\mathcal{A},({}^{\perp_{\mathcal{D}}}\mathcal{A})^{\perp_{ \mathcal{D}}})$ is a
$\mathcal{D}$-cotorsion pair in $\mathcal{C}$, called the $\mathcal{D}$-cotorsion pair \emph{generated} by
$\mathcal{A}$. The pair
$({}^{\perp_{\mathcal{D}}}(\mathcal{A}^{\perp_{\mathcal{D}}}),\mathcal{A}^{\perp_{ \mathcal{D}}})$ is a
$\mathcal{D}$-cotorsion pair in $\mathcal{C}$, called the $\mathcal{D}$-cotorsion pair \emph{cogenerated} by
$\mathcal{A}$.
\end{defi}

\begin{defi} \label{d:cot} \rm Let $\mathcal{C}$ be an additive category with an exact structure $\mathcal{D}$. A
$\mathcal{D}$-cotorsion pair $(\mathcal{A},\mathcal{B})$ in $\mathcal{C}$ is said to have: 
\begin{enumerate}
\item \emph{enough $\mathcal{D}$-injectives} if for every object $X$ of $\mathcal{C}$ there exists a
$\mathcal{D}$-conflation $X\rightarrowtail B\twoheadrightarrow A$ for some objects $A\in \mathcal{A}$ and $B\in
\mathcal{B}$. 
\item \emph{enough $\mathcal{D}$-projectives} if for every object $Z$ of $\mathcal{C}$ there exists a
$\mathcal{D}$-conflation $B\rightarrowtail A\twoheadrightarrow Z$ for some objects $A\in \mathcal{A}$ and $B\in
\mathcal{B}$. 
\end{enumerate}

Also, $(\mathcal{A},\mathcal{B})$ is called \emph{$\mathcal{D}$-complete} if $(\mathcal{A},\mathcal{B})$ has enough
$\mathcal{D}$-injectives and enough $\mathcal{D}$-projectives.
\end{defi}

\begin{remark} \rm (1) An additive category $\mathcal{C}$ with an exact structure $\mathcal{D}$ has enough
$\mathcal{D}$-injectives if the $\mathcal{D}$-cotorsion pair $(\mathcal{C},Inj(\mathcal{D}))$ has enough
$\mathcal{D}$-injectives, where $Inj(\mathcal{D})$ denotes the class of $\mathcal{D}$-injective objects of
$\mathcal{C}$. Dually, $\mathcal{C}$ has enough $\mathcal{D}$-projectives if the $\mathcal{D}$-cotorsion pair
$(Proj(\mathcal{D}),\mathcal{C})$ has enough $\mathcal{D}$-projectives, where $Proj(\mathcal{D})$ denotes the class of
$\mathcal{D}$-projective objects of $\mathcal{C}$.

(2) Let $\mathcal{C}$ be an additive category with an exact structure $\mathcal{D}$, and let
$(\mathcal{A},\mathcal{B})$ be a $\mathcal{D}$-cotorsion pair in $\mathcal{C}$. Then both $\mathcal{A}$ and
$\mathcal{B}$ are closed under $\mathcal{D}$-extensions \cite[15.3]{Montano}.
\end{remark}

\begin{defi} \rm Let $\mathcal{A}$ be a class of objects (always considered to be closed under isomorphisms) in a
category $\mathcal{C}$. A morphism $f:A\to X$ in $\mathcal{C}$ with $A\in \mathcal{A}$ is called an
\emph{$\mathcal{A}$-precover} of $X$ if every morphism $f':A'\to X$ in $\mathcal{C}$ with $A'\in \mathcal{A}$ factors
through $f$. An $\mathcal{A}$-precover $f:A\to X$ of $X$ is called an \emph{$\mathcal{A}$-cover} if every endomorphism
$g:A\to A$ with $fg=f$ is an automorphism. The class $\mathcal{A}$ is called \emph{(pre)covering} if every object $X$ of
$\mathcal{C}$ has an $\mathcal{A}$-(pre)cover. The concepts of \emph{$\mathcal{A}$-(pre)envelope} and
\emph{(pre)enveloping} class are defined dually.
\end{defi}

\begin{defi}\rm Let $\mathcal{C}$ be an additive category with an exact structure $\mathcal{D}$ and $\mathcal{A}$ a
class of objects in $\mathcal{C}$. 

An $\mathcal{A}$-\emph{precover relative to} $\mathcal{D}$ of an object $C\in
\mathcal{C}$ is a $\mathcal{D}$-deflation $f:X\twoheadrightarrow C$ with $X\in
\mathcal{A}$ such that every morphism $f':X'\to C$ in $\mathcal{C}$ with $X'\in \mathcal{A}$ factors through $f$. An
$\mathcal{A}$-{precover relative to} $\mathcal{D}$ of an object $C\in \mathcal{C}$ is called an
$\mathcal{A}$-\emph{cover relative to} $\mathcal{D}$ of $C$ if every endomorphism $g:X\to X$ with $fg=f$ is an
automorphism. The class $\mathcal{A}$ is called a $\mathcal{D}$-\emph{precovering class} ($\mathcal{D}$-\emph{covering
class}) if every object $C\in \mathcal{C}$ has an $\mathcal{A}$-precover ($\mathcal{A}$-cover) relative to
$\mathcal{D}$. The concepts of \emph{$\mathcal{A}$-(pre)envelope} relative to $\mathcal{D}$ and
$\mathcal{D}$-\emph{(pre)enveloping} class are defined dually.

A $\mathcal{D}$-cotorsion pair $(\mathcal{A},\mathcal{B})$ in $\mathcal{C}$ is called \emph{$\mathcal{D}$-perfect} if
$\mathcal{A}$ is a $\mathcal{D}$-covering class and $\mathcal{B}$ is a $\mathcal{D}$-enveloping class.
\end{defi}

The following result is easily obtained from \cite[Section~16]{Montano}, and shows how $\mathcal{D}$-enveloping and
enveloping ($\mathcal{D}$-covering and covering) classes are related.

\begin{prop} \cite[Propositions~16.3, 16.4]{Montano} 
Let $\mathcal{C}$ be an additive category with an exact structure $\mathcal{D}$, and let $\mathcal{A}$ be a class of
objects of $\mathcal{C}$.
\begin{enumerate} 
\item Assume that $\mathcal{D}$ is injectively generated. Then $\mathcal{A}$ is a $\mathcal{D}$-enveloping class if and
only if $\mathcal{A}$ is an enveloping class and $Inj(\mathcal{D})\subseteq \mathcal{A}$.
\item Assume that $\mathcal{D}$ is projectively generated. Then $\mathcal{A}$ is a $\mathcal{D}$-covering class if and
only if $\mathcal{A}$ is a covering class and $Proj(\mathcal{D})\subseteq \mathcal{A}$.
\end{enumerate}
\end{prop}

For an exhaustive account on exact categories, covers and envelopes we refer to \cite{Buhler,Xu}.

\section{Two Galois connections}

The main concepts studied in the paper are those of relatively divisible and relatively flat objects. They are inspired
by Sklyarenko's modules of flat type relative to a proper class \cite{Skly}. Let us recall their definition \cite[15.4,
15.11]{Montano}, adapted in an obvious way to exact categories. 

\begin{defi} \rm Let $\mathcal{C}$ be an additive category, and let $\mathcal{D}$ and $\mathcal{E}$ be exact structures
on $\mathcal{C}$. An object $X$ of $\mathcal{C}$ is called:
\begin{enumerate}
\item \emph{$\mathcal{D}$-$\mathcal{E}$-divisible} if every $\mathcal{D}$-inflation $X\rightarrowtail Y$ is an
$\mathcal{E}$-inflation.
\item \emph{$\mathcal{D}$-$\mathcal{E}$-flat} if every $\mathcal{D}$-deflation $Y\twoheadrightarrow X$
is an $\mathcal{E}$-deflation.
\end{enumerate}
We denote by $Div(\mathcal{D}\textrm{-}\mathcal{E})$ and $Flat(\mathcal{D}\textrm{-}\mathcal{E})$ the classes of
$\mathcal{D}$-$\mathcal{E}$-divisible and $\mathcal{D}$-$\mathcal{E}$-flat objects of $\mathcal{C}$ respectively.
\end{defi}

\begin{remark} \rm (1) Let $\mathcal{C}$ be an additive category with an exact structure $\mathcal{D}$. For the exact
structure $\mathcal{E}$ consisting of all split exact sequences, $\mathcal{D}$-$\mathcal{E}$-divisible and
$\mathcal{D}$-$\mathcal{E}$-flat objects coincide with $\mathcal{D}$-injective and $\mathcal{D}$-projective objects
respectively.  

(2) Let $\mathcal{C}$ be a quasi-abelian category. Let $\mathcal{D}$ be the exact structure on $\mathcal{C}$ given by
all kernel-cokernel pairs and $\mathcal{E}$ an exact structure on $\mathcal{C}$. Then
$\mathcal{D}$-$\mathcal{E}$-divisible objects and $\mathcal{D}$-$\mathcal{E}$-flat objects of $\mathcal{C}$ are simply
called \emph{$\mathcal{E}$-divisible} and \emph{$\mathcal{E}$-flat} respectively. We denote by
$Div(\mathcal{E})$ and $Flat(\mathcal{E})$ the classes of $\mathcal{E}$-divisible and $\mathcal{E}$-flat
objects of $\mathcal{C}$ respectively.
\end{remark}

In what follows we shall show that relatively divisible and relatively flat objects may be used for establishing Galois
connections between the posets of relative cotorsion pairs, relative projectively generated exact structures and
relative injectively generated exact structures. We extend and reformulate results from \cite[Section~14]{Montano} in
terms of Galois connections.

We begin with the following proposition, which extends \cite[15.8, 15.15]{Montano} from abelian categories to exact categories.

\begin{prop} \label{p:Ext} Let $\mathcal{C}$ be an additive category, and let $\mathcal{D}$ and $\mathcal{E}$ be exact
structures on $\mathcal{C}$.
\begin{enumerate}
\item Assume that $\mathcal{E}$ is $\mathcal{D}$-projectively generated by a class $\mathcal{M}$ of objects of
$\mathcal{C}$. Then an object $X$ of $\mathcal{C}$ is $\mathcal{D}$-$\mathcal{E}$-divisible if and only if 
${\rm Ext}^1_{\mathcal{D}}(\mathcal{M},X)=0$.  
\item Assume that $\mathcal{E}$ is $\mathcal{D}$-injectively generated by a class $\mathcal{M}$ of objects of
$\mathcal{C}$. Then an object $Z$ of $\mathcal{C}$ is $\mathcal{D}$-$\mathcal{E}$-flat if and only if ${\rm
Ext}^1_{\mathcal{D}}(Z,\mathcal{M})=0$.  
\end{enumerate}
\end{prop}

\begin{proof} (1) (i)$\Rightarrow$(ii) Assume that $X$ is $\mathcal{D}$-$\mathcal{E}$-divisible. Let $X\rightarrowtail 
Y\twoheadrightarrow M$ be a $\mathcal{D}$-conflation with $M\in \mathcal{M}$. By hypothesis, it must be an
$\mathcal{E}$-conflation. Since $\mathcal{E}=\mathcal{E}^{\mathcal{M}}_H$, $M$ is projective with respect to this
$\mathcal{E}$-conflation, which implies its splitness. Hence ${\rm Ext}^1_{\mathcal{D}}(\mathcal{M},X)=0$.

(ii)$\Rightarrow$(i) Assume that ${\rm Ext}^1_{\mathcal{D}}(\mathcal{M},X)=0$. Let $X\rightarrowtail Y\twoheadrightarrow
Z$ be a $\mathcal{D}$-conflation. Since this $\mathcal{D}$-conflation splits, every $M\in \mathcal{M}$ is projective
with respect to the above $\mathcal{D}$-conflation. Then the $\mathcal{D}$-conflation must be an
$\mathcal{E}$-conflation, because $\mathcal{E}=\mathcal{E}^{\mathcal{M}}_H$. This shows that $X$ is
$\mathcal{D}$-$\mathcal{E}$-divisible.
\end{proof}

The next lemma can be immediately deduced from Proposition \ref{p:Ext} (also, see \cite [Lemmas~14.8, 14.9]{Montano}).

\begin{lem} \label{l:simplif} Let $\mathcal{C}$ be an additive category, and let $\mathcal{D}$ and $\mathcal{E}$ be
exact structures on $\mathcal{C}$. 
\begin{enumerate}
\item Assume that $\mathcal{E}$ is $\mathcal{D}$-projectively generated. Then
$({}^{\perp_{\mathcal{D}}}Div(\mathcal{D}\textrm{-}\mathcal{E}))^{\perp_{\mathcal{D}}}=Div(\mathcal{D}\textrm{-}\mathcal
{E})$. 
\item Assume that $\mathcal{E}$ is $\mathcal{D}$-injectively generated. Then
${}^{\perp_{\mathcal{D}}}(Flat(\mathcal{D}\textrm{-}\mathcal{E})^{\perp_{\mathcal{D}}})=Flat(\mathcal{D}\textrm{-}
\mathcal{E})$.
\end{enumerate}
\end{lem}

Let us recall the concept of Galois connection between posets (e.g., see \cite{Erne}). 

\begin{defi} \rm Let $(A,\leq)$ and $(B,\leq)$ be posets. 
\begin{enumerate}
\item A {\it monotone Galois connection} between $(A,\leq)$ and $(B,\leq)$ consists of a pair $(\alpha, \beta)$ of two
order-preserving maps $\alpha:A\to B$ and $\beta:B\to A$ such that for all $a\in A$ and $b\in B$, we have
$\alpha(a)\leq b$ if and only if $a\leq \beta(b)$. Equivalently, $(\alpha,\beta)$ is a monotone Galois connection if and
only if for all $a\in A$, $a\leq \beta\alpha(a)$ and for all $b\in B$, $\alpha\beta(b)\leq b$. An element $a\in A$
(respectively $b\in B$) is called \emph{Galois} with respect to $(\alpha, \beta)$ if $\beta\alpha(a)=a$ (respectively
$\alpha\beta(b)=b$). 
\item An {\it antitone Galois connection} between $(A,\leq)$ and $(B,\leq)$ is a monotone Galois connection between
$(A,\leq)$ and $(B,\geq)$. 
\end{enumerate}
\end{defi}

Let $\mathcal{C}$ be an additive category with an exact structure $\mathcal{D}$. The class $Ex(\mathcal{C})$ of exact
structures on $\mathcal{C}$ is partially ordered as follows: for every $\mathcal{E}_1,\mathcal{E}_2\in Ex(\mathcal{C})$,
$\mathcal{E}_1\supseteq \mathcal{E}_2$ if and only if every $\mathcal{E}_2$-conflation
is an $\mathcal{E}_1$-conflation. This relation restricts to ones on the classes $DPEx(\mathcal{C})$ and
$DIEx(\mathcal{C})$ of $\mathcal{D}$-projectively and $\mathcal{D}$-injectively generated exact structures on
$\mathcal{C}$. The class $DCot(\mathcal{C})$ of $\mathcal{D}$-cotorsion pairs in $\mathcal{C}$ is partially ordered as
follows: for every $(\mathcal{A}_1,\mathcal{B}_1), (\mathcal{A}_2,\mathcal{B}_2)\in DCot(\mathcal{C})$,
$(\mathcal{A}_1,\mathcal{B}_1)\geq (\mathcal{A}_2,\mathcal{B}_2)$ if and only if $\mathcal{A}_2\supseteq \mathcal{A}_1$
if and only if $\mathcal{B}_1\supseteq \mathcal{B}_2$ \cite{GSW}.

Following \cite[Section~14]{Montano} and the original work by Salce \cite{Salce}, we consider the following maps. 
\begin{enumerate}
\item Let $\Psi:DPEx(\mathcal{C})\to DCot(\mathcal{C})$ be defined by
$$\Psi(\mathcal{E})=({}^{\perp_{\mathcal{D}}}Div(\mathcal{D}\textrm{-}\mathcal{E}),Div(\mathcal{D}\textrm{-}\mathcal{E}
))=({}^{\perp_{\mathcal{D}}}Div(\mathcal{D}\textrm{-}\mathcal{E}),({}^{\perp_{\mathcal{D}}}Div(\mathcal{D}\textrm{-}
\mathcal{E}))^{\perp_{ \mathcal{D}}})$$ (see Lemma \ref{l:simplif}) for every $\mathcal{E}\in DPEx(\mathcal{C})$,
that is, the $\mathcal{D}$-cotorsion pair generated by $Div(\mathcal{D}\textrm{-}\mathcal{E})$. Let
$\overset{\sim}\Psi:DCot(\mathcal{C})\to DPEx(\mathcal{C})$ be defined by
$\overset{\sim}\Psi(\mathcal{A},\mathcal{B})=\pi_{\mathcal{D}}^{-1}(\mathcal{A})$ for every
$(\mathcal{A},\mathcal{B})\in DCot(\mathcal{C})$, that is, the exact structure $\mathcal{D}$-projectively generated by
$\mathcal{A}$.
\item Let $\Phi:DIEx(\mathcal{C})\to DCot(\mathcal{C})$ be defined by
$$\Phi(\mathcal{E})=(Flat(\mathcal{D}\textrm{-}\mathcal{E}),Flat(\mathcal{D}\textrm{-}\mathcal{E})^{\perp_{\mathcal{D}}}
)=({}^{\perp_{\mathcal{D}}}(Flat(\mathcal{D}\textrm{-}\mathcal{E})^{\perp_{\mathcal{D}}}),Flat(\mathcal{D}\textrm{-}
\mathcal{E})^{\perp_{ \mathcal{D}}})$$ (see Lemma \ref{l:simplif}) for every $\mathcal{E}\in DIEx(\mathcal{C})$, that
is, the $\mathcal{D}$-cotorsion pair cogenerated by $Flat(\mathcal{E})$. Let
$\overset{\sim}\Phi:DCot(\mathcal{C})\to DIEx(\mathcal{C})$ be defined by
$\overset{\sim}\Phi(\mathcal{A},\mathcal{B})=\iota_{\mathcal{D}}^{-1}(\mathcal{B})$ for every
$(\mathcal{A},\mathcal{B})\in DCot(\mathcal{C})$, that is, the exact structure $\mathcal{D}$-injectively generated by
$\mathcal{B}$.
\end{enumerate}

The following proposition offers one of the main motivations for studying relative divisibility and relative flatness. It is 
an immediate generalization of \cite[15.10, 15.17]{Montano} from abelian categories to
exact categories, and it will be frequently used in what follows, sometimes without explicit reference.

\begin{prop} \label{p:identif} With the above notation:
\begin{enumerate}
\item The map $\Psi$ is surjective. More precisely, for every $\mathcal{D}$-cotorsion pair $(\mathcal{A},\mathcal{B})$
in $\mathcal{C}$, $\mathcal{B}=Div(\mathcal{D}\textrm{-}\mathcal{E})$ for the $\mathcal{D}$-projectively generated
exact structure $\mathcal{E}=\pi^{-1}_{\mathcal{D}}(\mathcal{A})$ on $\mathcal{C}$.
\item The map $\Phi$ is surjective. More precisely, for every $\mathcal{D}$-cotorsion pair $(\mathcal{A},\mathcal{B})$
in $\mathcal{C}$, $\mathcal{A}=Flat(\mathcal{D}\textrm{-}\mathcal{E})$ for the
$\mathcal{D}$-injectively generated exact structure $\mathcal{E}=\iota^{-1}_{\mathcal{D}}(\mathcal{B})$ on
$\mathcal{C}$.
\end{enumerate}
\end{prop}

\begin{theor} \label{t:Galois} With the above notation:
\begin{enumerate}
\item The pair $(\Psi,\overset{\sim}\Psi)$ is a monotone Galois connection between the posets
$(DPEx(\mathcal{C}),\supseteq)$ and $(DCot(\mathcal{C}),\geq)$. 
\item The pair $(\Phi,\overset{\sim}\Phi)$ is an antitone Galois connection between the posets
$(DIEx(\mathcal{C}),\supseteq)$ and $(DCot(\mathcal{C}),\geq)$. 
\end{enumerate}
\end{theor}

\begin{proof} (1) One shows that $\Psi$ and $\overset{\sim}\Psi$ are order-preserving maps between the posets 
$(DPEx(\mathcal{C}),\supseteq)$ and $(DCot(\mathcal{C}),\geq)$ as in \cite[14.10, 14.14]{Montano}. Also,
$\mathcal{E}\supseteq \overset{\sim}\Psi(\Psi(\mathcal{E}))$ for every $\mathcal{E}\in DPEx(\mathcal{C})$, and
$\Psi(\overset{\sim}\Psi(\mathcal{A},\mathcal{B}))\geq (\mathcal{A},\mathcal{B})$ (even equality) for every
$(\mathcal{A},\mathcal{B})\in DCot(\mathcal{C})$ as in \cite[14.15]{Montano}. 

(2) One shows that $\Phi$ and $\overset{\sim}\Phi$ are order-reversing maps between the
posets $(DIEx(\mathcal{C}),\supseteq)$ and $(DCot(\mathcal{C}),\geq)$ as in \cite[14.10, 14.14]{Montano}. Also,
$\mathcal{E}\supseteq \overset{\sim}\Phi(\Phi(\mathcal{E}))$ for every $\mathcal{E}\in DIEx(\mathcal{C})$, and
$(\mathcal{A},\mathcal{B})\geq \Phi(\overset{\sim}\Phi(\mathcal{A},\mathcal{B}))$ (even equality) for every
$(\mathcal{A},\mathcal{B})\in DCot(\mathcal{C})$ as in \cite[14.15]{Montano}.
\end{proof}

The following definition extends the notion of injectively generated Xu exact structure \cite[14.16]{Montano} from
abelian categories to exact categories, and also gives its dual.

\begin{defi} \rm Let $\mathcal{C}$ be an additive category with an exact structure $\mathcal{D}$ .
\begin{enumerate}
\item A $\mathcal{D}$-projectively generated exact structure $\mathcal{E}$ on $\mathcal{C}$ is called a \emph{Xu exact
structure} if $Proj(\mathcal{E})={}^{\perp_{\mathcal{D}}}Div(\mathcal{D}\textrm{-}\mathcal{E})$.
\item A $\mathcal{D}$-injectively generated exact structure $\mathcal{E}$ on $\mathcal{C}$ is called a \emph{Xu exact
structure} if $Inj(\mathcal{E})=Flat(\mathcal{D}\textrm{-}\mathcal{E})^{\perp_{\mathcal{D}}}$.
\end{enumerate}
\end{defi}

\begin{cor} Let $\mathcal{C}$ be an additive category with an exact structure $\mathcal{D}$. Then there are bijective
correspondences between:
\begin{enumerate}
\item $\mathcal{D}$-cotorsion pairs in $\mathcal{C}$. 
\item $\mathcal{D}$-projectively generated Xu exact structures on $\mathcal{C}$.
\item $\mathcal{D}$-injectively generated Xu exact structures on $\mathcal{C}$.
\end{enumerate}
\end{cor}

\begin{proof} We shall use the following well-known result on Galois connections: if a pair $(\alpha, \beta)$ of maps
$\alpha:A\to B$ and $\beta:B\to A$ is a monotone Galois connection between posets $(A,\leq)$ and $(B,\leq)$, then
$\alpha$ and $\beta$ induce a bijective correspondence between the Galois elements of $A$ and the Galois elements of $B$
with respect to $(\alpha, \beta)$.

By the proof of Theorem \ref{t:Galois}, every $\mathcal{D}$-cotorsion pair in $\mathcal{C}$ is a Galois element with
respect to both Galois connections $(\Psi,\overset{\sim}\Psi)$ and $(\Phi,\overset{\sim}\Phi)$. 

Let us show that the Galois elements of $DPEx(\mathcal{C})$ with respect to the Galois connection
$(\Psi,\overset{\sim}\Psi)$ are the $\mathcal{D}$-projectively generated Xu exact structures on $\mathcal{C}$. First,
let $\mathcal{E}$ be a Galois element of $DPEx(\mathcal{C})$ with respect to $(\Psi,\overset{\sim}\Psi)$. Then
$\Psi(\mathcal{E})=(\mathcal{A},\mathcal{B})$ for some $\mathcal{D}$-cotorsion pair $(\mathcal{A},\mathcal{B})$ in
$\mathcal{C}$, and we have
$$\mathcal{E}=\overset{\sim}\Psi(\Psi(\mathcal{E}))=\overset{\sim}\Psi(\mathcal{A},\mathcal{B})=\pi^{-1}_{\mathcal{D}}
(\mathcal{A}).$$ By Proposition \ref{p:identif}, it follows that
${}^{\perp_{\mathcal{D}}}Div(\mathcal{D}\textrm{-}\mathcal{E})={}^{\perp_{\mathcal{D}}}\mathcal{B}=\mathcal{A}\subseteq
Proj(\pi^{-1}_{\mathcal{D}}(\mathcal{A}))=Proj(\mathcal{E})$.
On the other hand, we have $Proj(\mathcal{E})\subseteq {}^{\perp_{\mathcal{D}}}Div(\mathcal{D}\textrm{-}\mathcal{E})$,
since $\mathcal{E}$ is $\mathcal{D}$-projectively generated. Hence $\mathcal{E}$ is a $\mathcal{D}$-projectively
generated Xu exact structure on $\mathcal{C}$. Conversely, let $\mathcal{E}$ be a $\mathcal{D}$-projectively generated
Xu exact structure on $\mathcal{C}$. Then
$$\mathcal{E}=\pi^{-1}_{\mathcal{D}}(Proj(\mathcal{E}))=\pi^{-1}_{\mathcal{D}}({}^{\perp_{\mathcal{D}}}Div(\mathcal{D}
\textrm{-}\mathcal{E}))=\overset{\sim}\Psi(\Psi(\mathcal{E})),$$ hence $\mathcal{E}$ is a Galois element of
$DPEx(\mathcal{C})$ with respect to $(\Psi,\overset{\sim}\Psi)$.

Similarly, one shows that the Galois elements of $DIEx(\mathcal{C})$ with respect to the Galois connection
$(\Phi,\overset{\sim}\Phi)$ are the $\mathcal{D}$-injectively generated Xu exact structures on $\mathcal{C}$.
\end{proof}

We end this section with a characterization of injectively and projectively generated Xu exact structures, which
generalizes \cite[Theorem~3.5.1]{Xu} from module categories to exact categories, and also gives its dual. The arguments 
of the implications (ii)$\Rightarrow$(iii) extend those of the Wakamatsu Lemmas \cite[Lemmas 2.1.1 and 2.1.2]{Xu}. 

\begin{theor} \label{t:extensions} Let $\mathcal{C}$ be an additive category, and let $\mathcal{D}$ and $\mathcal{E}$ be
exact structures on $\mathcal{C}$. 
\begin{enumerate}
\item Assume that $\mathcal{E}$ is $\mathcal{D}$-projectively generated and every object $Z$ of $\mathcal{C}$ has a
$Proj(\mathcal{E})$-cover $\mathcal{P}(Z)\twoheadrightarrow Z$ relative to $\mathcal{D}$. Then the following are
equivalent:
\begin{enumerate}[(i)]
\item $\mathcal{E}$ is a $\mathcal{D}$-projectively generated Xu exact structure on $\mathcal{C}$.
\item $Proj(\mathcal{E})$ is closed under $\mathcal{D}$-extensions.
\item For every $\mathcal{D}$-conflation $X\rightarrowtail \mathcal{P}(Z)\twoheadrightarrow Z$ in $\mathcal{C}$, $X$ is
$\mathcal{D}$-$\mathcal{E}$-divisible.
\end{enumerate}
\item Assume that $\mathcal{E}$ is $\mathcal{D}$-injectively generated and every object $X$ of $\mathcal{C}$ has an
$Inj(\mathcal{E})$-envelope $X\rightarrowtail\mathcal{I}(X)$ relative to $\mathcal{D}$. Then the following are
equivalent:
\begin{enumerate}[(i)]
\item $\mathcal{E}$ is a $\mathcal{D}$-injectively generated Xu exact structure on $\mathcal{C}$.
\item $Inj(\mathcal{E})$ is closed under $\mathcal{D}$-extensions.
\item For every $\mathcal{D}$-conflation $X\rightarrowtail \mathcal{I}(X)\twoheadrightarrow Z$ in $\mathcal{C}$, $Z$ is
$\mathcal{D}$-$\mathcal{E}$-flat.
\end{enumerate}
\end{enumerate}
\end{theor}

\begin{proof} (1) (i)$\Rightarrow$(ii) Assume that (i) holds. Then
$Proj(\mathcal{E})={}^{\perp_{\mathcal{D}}}Div(\mathcal{D}\textrm{-}\mathcal{E})$ is closed under
$\mathcal{D}$-extensions, as a class of the above $\mathcal{D}$-cotorsion pair $\Psi(\mathcal{E})$.
 
(ii)$\Rightarrow$(iii) Assume that (ii) holds. Consider a $\mathcal{D}$-conflation $X\stackrel{i}\rightarrowtail
\mathcal{P}(Z)\stackrel{d}\twoheadrightarrow Z$. Since $\mathcal{E}$ is $\mathcal{D}$-projectively generated, we have
$\mathcal{E}=\pi^{-1}_{\mathcal{D}}(Proj(\mathcal{E}))$. Then by Proposition \ref{p:Ext}, in order to show that $X$ is
$\mathcal{D}$-$\mathcal{E}$-divisible, it is enough to prove that ${\rm Ext}^1_{\mathcal{D}}(M,X)=0$ for every $M\in
Proj(\mathcal{E})$. To this end, let $M\in Proj(\mathcal{E})$, and consider a $\mathcal{D}$-conflation
$X\stackrel{f}\rightarrowtail Y\stackrel{g}\twoheadrightarrow M$. Taking the pushout of the morphisms $i$ and $f$ we may
construct (see Lemma \ref{l:POPB}) the following commutative diagram:
$$\SelectTips{cm}{}
\xymatrix{
X \ar@{>->}[r]^i \ar@{>->}[d]_f & \mathcal{P}(Z) \ar@{->>}[r]^d \ar@{>->}[d]^{f'} & Z \ar@{=}[d] \\
Y \ar@{>->}[r]_{i'} \ar@{->>}[d]_g & V \ar@{->>}[r]_{d'} \ar@{->>}[d]^{g'} & Z \\
M \ar@{=}[r] & M &  
}$$
in which the rows and the columns are $\mathcal{D}$-conflations. Since $Proj(\mathcal{E})$ is closed under 
$\mathcal{D}$-extensions, we have $V\in Proj(\mathcal{E})$. Now the $Proj(\mathcal{E})$-precover property of
$\mathcal{P}(Z)$ yields a morphism $h:V\to \mathcal{P}(Z)$ such that $d'=dh$. Then $d=dhf'$, and the
$Proj(\mathcal{E})$-cover property of $\mathcal{P}(Z)$ implies that $hf'$ is an automorphism of $\mathcal{P}(Z)$.
Denote $u=(hf')^{-1}hi':Y\to \mathcal{P}(Z)$. Then $uf=i$ and we may construct the following commutative diagram:
$$\SelectTips{cm}{}
\xymatrix{
X \ar@{>->}[r]^f \ar@{=}[d] & Y \ar@{->>}[r]^g \ar[d]^u & M \ar[d]^v \\
X \ar@{>->}[r]_i & \mathcal{P}(Z) \ar@{->>}[r]_d & Z 
}$$
The $Proj(\mathcal{E})$-precover property of $\mathcal{P}(Z)$ yields a morphism $w:M\to \mathcal{P}(Z)$ such that
$dw=v$. Then the $\mathcal{D}$-conflation $X\stackrel{f}\rightarrowtail Y\stackrel{g}\twoheadrightarrow M$ splits by the
Homotopy Lemma \cite[7.16]{Wisb}. Therefore, ${\rm Ext}^1_{\mathcal{D}}(M,X)=0$ for every $M\in Proj(\mathcal{E})$,
which shows that $X$ is $\mathcal{D}$-$\mathcal{E}$-divisible.

(iii)$\Rightarrow$(i) Assume that (iii) holds. Let $Z\in
{}^{\perp_{\mathcal{D}}}Div(\mathcal{D}\textrm{-}\mathcal{E})$. Consider the $\mathcal{D}$-conflation $X\rightarrowtail
\mathcal{P}(Z)\twoheadrightarrow Z$. Then $X\in Div(\mathcal{D}\textrm{-}\mathcal{E})$, and so the
$\mathcal{D}$-conflation splits. Hence $Z\in Proj(\mathcal{E})$, which shows that $\mathcal{E}$ is a
$\mathcal{D}$-projectively generated Xu exact structure.
\end{proof}

\section{Properties of relatively divisible and relatively flat objects}

We collect in the next proposition some first results on the classes of relatively divisible and relatively
flat objects in exact categories. They are immediate generalizations of \cite[15.7 and 15.14]{Montano}) from abelian
categories to exact categories. 

\begin{prop} \label{p:closureDE} Let $\mathcal{C}$ be an additive category, and let $\mathcal{D}$ and $\mathcal{E}$ be
exact structures on $\mathcal{C}$.  
\begin{enumerate} \item 
\begin{enumerate}[(i)] 
\item The class of $\mathcal{D}$-$\mathcal{E}$-divisible objects of $\mathcal{C}$ is closed under
$\mathcal{D}$-extensions and $\mathcal{E}$-inflations. 
\item Assume that $\mathcal{C}$ has products and $\mathcal{E}$ is closed under products. Then the class of
$\mathcal{D}$-$\mathcal{E}$-divisible objects of $\mathcal{C}$ is closed under products.
\item Assume that $\mathcal{C}$ has direct limits and $\mathcal{E}$ is closed under direct limits. Let $(X_i,f_{ij})_I$
be a direct system of $\mathcal{D}$-$\mathcal{E}$-divisible objects of $\mathcal{C}$ with direct limit
$(\underset{\longrightarrow}\lim \,X_i,f_i)$ such that each $f_{ij}$ is a $\mathcal{D}$-inflation. Then
$\underset{\longrightarrow}\lim \,X_i$ is $\mathcal{D}$-$\mathcal{E}$-divisible. In particular, the class of 
$\mathcal{D}$-$\mathcal{E}$-divisible objects of $\mathcal{C}$ is closed under coproducts.
\item Every object of $\mathcal{C}$ is $\mathcal{D}$-$\mathcal{E}$-divisible if and only if $\mathcal{D}\subseteq
\mathcal{E}$.
\item Assume that $\mathcal{E}$ is $\mathcal{D}$-projectively generated by a class $\mathcal{M}$ of objects of
$\mathcal{C}$. Then every object of $\mathcal{C}$ is $\mathcal{D}$-$\mathcal{E}$-divisible if and only if every object
of $\mathcal{M}$ is $\mathcal{D}$-projective. 
\end{enumerate}
\item 
\begin{enumerate}[(i)]
\item The class of $\mathcal{D}$-$\mathcal{E}$-flat objects of $\mathcal{C}$ is closed under $\mathcal{D}$-extensions
and $\mathcal{E}$-deflations. 
\item Assume that $\mathcal{C}$ has coproducts and $\mathcal{E}$ is closed under coproducts. Then the class of
$\mathcal{D}$-$\mathcal{E}$-flat objects of $\mathcal{C}$ is closed under coproducts.
\item Assume that $\mathcal{C}$ has direct limits and $\mathcal{E}$ is closed under direct limits. Then the class of
$\mathcal{D}$-$\mathcal{E}$-flat objects of $\mathcal{C}$ is closed under direct limits.
\item Every object of $\mathcal{C}$ is $\mathcal{D}$-$\mathcal{E}$-flat if and only if $\mathcal{D}\subseteq
\mathcal{E}$.
\item Assume that $\mathcal{E}$ is $\mathcal{D}$-injectively generated by a class $\mathcal{M}$ of objects of
$\mathcal{C}$. Then every object of $\mathcal{C}$ is $\mathcal{D}$-$\mathcal{E}$-flat if and only if every object of
$\mathcal{M}$ is $\mathcal{D}$-injective. 
\end{enumerate}
\end{enumerate}
\end{prop}

The following proposition is one of the key results on relatively divisible and relatively flat objects in exact
categories, having a number of important consequences.

\begin{prop} \label{p:diagram} Let $\mathcal{C}$ be an additive category, and let $\mathcal{D}$ and $\mathcal{E}$
be exact structures on $\mathcal{C}$. Consider the following commutative diagram in $\mathcal{C}$:
\[\SelectTips{cm}{}
\xymatrix{
X \ar@{=}[d] \ar@{>->}[r]^i & Y \ar[d]^f \ar@{->>}[r]^d & Z \ar[d]^g \ar@{>->}[r]^j & U \ar[d]^h \ar@{->>}[r]^p & V
\ar@{=}[d] \\
X \ar@{>->}[r]_{i'} & Y' \ar@{->>}[r]_{d'} & Z' \ar@{>->}[r]_{j'} & U' \ar@{->>}[r]_{p'} & V 
}\]
where each row consists of two $\mathcal{D}$-conflations. 
\begin{enumerate}
\item Assume that ${\rm Ext}^1_{\mathcal{D}}(Z,X)=0$ and $Y'$ is $\mathcal{D}$-$\mathcal{E}$-divisible. Then $Z'$ is
$\mathcal{D}$-$\mathcal{E}$-divisible.
\item Assume that ${\rm Ext}^1_{\mathcal{D}}(V,Z')=0$ and $U$ is $\mathcal{D}$-$\mathcal{E}$-flat. Then $Z$ is
$\mathcal{D}$-$\mathcal{E}$-flat.
\end{enumerate}
\end{prop}

\begin{proof} (1) Consider the following commutative diagram: 
\[\SelectTips{cm}{}
\xymatrix{
Y'\oplus Z \ar@{>->}[r]^{\left [\begin{smallmatrix} 1&0 \\ 0&j \end{smallmatrix}\right ]} \ar[d]_-{\left
[\begin{smallmatrix} d'&g \end{smallmatrix}\right ]} & Y'\oplus U \ar@{->>}[r]^-{\left [\begin{smallmatrix} 0&p
\end{smallmatrix}\right ]} \ar[d]_-{\left [\begin{smallmatrix} j'd'&h \end{smallmatrix}\right ]} & V \ar@{=}[d] \\
Z' \ar@{>->}[r]_{j'} & U' \ar@{->>}[r]_{p'} & V  
}\]
where the rows are $\mathcal{D}$-conflations. Then the left square is a pushout by Lemma \ref{l:POPB}. 

Since ${\rm Ext}^1_{\mathcal{D}}(Z,X)=0$, there exists a morphism $\alpha:Z\to Y'$ such that $d'\alpha=g$. Consider
the pushout of the morphisms $\left [\begin{smallmatrix} 1 & \alpha \end{smallmatrix}\right ]:Y'\oplus Z\to Y'$ 
and $\left [\begin{smallmatrix} 1&0 \\ 0&j \end{smallmatrix}\right ]:Y'\oplus Z\to Y'\oplus U$, and denote by $\left
[\begin{smallmatrix} w_1 & w_2 \end{smallmatrix}\right ]:Y'\oplus U\to W$ and $l:Y'\to W$ the resulting morphisms.
Since $j' \left [\begin{smallmatrix} d' & g \end{smallmatrix}\right ]=\left [\begin{smallmatrix} j'd' & h
\end{smallmatrix}\right ] \left [\begin{smallmatrix} 1&0 \\ 0&j \end{smallmatrix}\right ]$, the pushout property implies
the existence of a unique morphism $q:W\to U'$ such that $ql=j'd'$ and $q \left [\begin{smallmatrix} w_1 & w_2
\end{smallmatrix}\right ]=\left [\begin{smallmatrix} j'd' & h \end{smallmatrix}\right ]$. Hence we have the following
commutative diagram:
\[\SelectTips{cm}{}
\xymatrix{
Y'\oplus Z \ar[d]_{\left [\begin{smallmatrix} 1&0 \\ 0&j \end{smallmatrix}\right ]} \ar[r]^-{\left [\begin{smallmatrix}
1 & \alpha \end{smallmatrix}\right ]} & Y' \ar[d]^l \ar[r]^{d'} & Z' \ar[d]^{j'} \\
Y'\oplus U \ar[r]_-{\left [\begin{smallmatrix} w_1 & w_2 \end{smallmatrix}\right ]} & W \ar[r]_q & U' 
}\]
where the exterior rectangle is a pushout by the first part of the proof. Then the right square is a pushout by Lemma
\ref{l:rec} and we obtain the following commutative diagram:
\[\SelectTips{cm}{}
\xymatrix{
X \ar@{=}[d] \ar@{>->}[r]^{i'} & Y' \ar@{>->}[d]^l \ar@{->>}[r]^{d'} & Z' \ar@{>->}[d]^{j'} \\
X \ar@{>->}[r]_w & W \ar@{->>}[d]_v \ar@{->>}[r]_q & U' \ar@{->>}[d]^{p'} \\
 & V \ar@{=}[r] & V 
}\]
where the rows and the columns are $\mathcal{D}$-conflations. Since $Y'$ is $\mathcal{D}$-$\mathcal{E}$-divisible, the
first column is an $\mathcal{E}$-conflation. Then the last column is also an $\mathcal{E}$-conflation. Hence
$Z'$ is $\mathcal{D}$-$\mathcal{E}$-divisible.
\end{proof}

\begin{theor} \label{t:cohcot} Let $\mathcal{C}$ be an additive category, let $\mathcal{D}$ and $\mathcal{E}$ be exact
structures on $\mathcal{C}$, and $(\mathcal{A},\mathcal{B})$ a $\mathcal{D}$-cotorsion pair. Assume that $\mathcal{C}$
has enough $\mathcal{D}$-injectives and enough $\mathcal{D}$-projectives.
\begin{enumerate}
\item Assume that $\mathcal{E}$ is $\mathcal{D}$-projectively generated by a class $\mathcal{M}$ of objects of
$\mathcal{C}$. Then the following are equivalent:
\begin{enumerate}[(i)]
\item For every $\mathcal{D}$-conflation $X\rightarrowtail Y'\twoheadrightarrow Z'$ in $\mathcal{C}$ with $X\in
\mathcal{B}$ and $Y'$ $\mathcal{D}$-$\mathcal{E}$-divisible, $Z'$ is $\mathcal{D}$-$\mathcal{E}$-divisible.
\item For every $\mathcal{D}$-conflation $Z\rightarrowtail U\twoheadrightarrow V$ in $\mathcal{C}$ with $V\in
\mathcal{M}$ and $U$ $\mathcal{D}$-projective, $Z\in \mathcal{A}$.
\end{enumerate}
\item Assume that $\mathcal{E}$ is $\mathcal{D}$-injectively generated by a class $\mathcal{M}$ of objects of
$\mathcal{C}$. Then the following are equivalent:
\begin{enumerate}[(i)]
\item For every $\mathcal{D}$-conflation $Z\rightarrowtail U\twoheadrightarrow V$ in $\mathcal{C}$ with $V\in
\mathcal{A}$ and $U$ $\mathcal{D}$-$\mathcal{E}$-flat, $Z$ is $\mathcal{D}$-$\mathcal{E}$-flat.
\item For every $\mathcal{D}$-conflation $X\rightarrowtail Y'\twoheadrightarrow Z'$ in $\mathcal{C}$ with $X\in
\mathcal{M}$ and $Y'$ $\mathcal{D}$-injective, $Z'\in \mathcal{B}$.
\end{enumerate}
\end{enumerate}
\end{theor}

\begin{proof} (1) (i)$\Rightarrow$(ii) Let $Z\stackrel{j}\rightarrowtail U\stackrel{p}\twoheadrightarrow V$ be a
$\mathcal{D}$-conflation with $V\in \mathcal{M}$ and $U$ $\mathcal{D}$-projective. Let $X\stackrel{i}\rightarrowtail
Y\stackrel{d}\twoheadrightarrow Z$ be a $\mathcal{D}$-conflation with $X\in \mathcal{B}$. Consider a
$\mathcal{D}$-inflation $f:Y\to Y'$ for some $\mathcal{D}$-injective object $Y'$ of $\mathcal{C}$, and denote by
$i'=fi:X\to Y'$ the composed $\mathcal{D}$-inflation and by $d':Y'\to Z'$ its cokernel. Then there is a morphism $g:Z\to
Z'$ such that $gd=d'f$. Now consider the pushout of the $\mathcal{D}$-inflation $j$ and the morphism $g$ to obtain the
commutative diagram from Proposition \ref{p:diagram}. By hypothesis, $Z'$ is $\mathcal{D}$-$\mathcal{E}$-divisible,
hence we have ${\rm Ext}^1_{\mathcal{D}}(V,Z')=0$ by Proposition \ref{p:Ext}. Since $U$ is $\mathcal{D}$-projective,
we have ${\rm Ext}^1_{\mathcal{D}}(U,X)=0$. It follows that ${\rm Ext}^1_{\mathcal{D}}(Z,X)=0$ by Proposition
\ref{p:diagram}. Hence $Z\in \mathcal{A}$. 

(ii)$\Rightarrow$(i) Let $X\stackrel{i'}\rightarrowtail Y'\stackrel{d'}\twoheadrightarrow Z'$ be a
$\mathcal{D}$-conflation with $X\in \mathcal{B}$ and $Y'$ $\mathcal{D}$-$\mathcal{E}$-divisible. Let
$Z'\stackrel{j'}\rightarrowtail U'\stackrel{p'}\twoheadrightarrow V$ be a $\mathcal{D}$-conflation with $V\in
\mathcal{M}$. Consider a $\mathcal{D}$-deflation $h:U\to U'$ for some $\mathcal{D}$-projective object $U$ of
$\mathcal{C}$, and denote by $p=p'h:U\to V$ the composed $\mathcal{D}$-deflation and by $j:Z\to U$ its kernel. Then
there is a morphism $g:Z\to Z'$ such that $j'g=hj$. Now consider the pullback of the $\mathcal{D}$-deflation $d'$ and
the morphism $g$ to obtain the commutative diagram from Proposition \ref{p:diagram}. Since $Z\in \mathcal{A}$ and $X\in
\mathcal{B}$, we have ${\rm Ext}^1_{\mathcal{D}}(Z,X)=0$. Then $Z'$ is $\mathcal{D}$-$\mathcal{E}$-divisible by
Proposition \ref{p:diagram}.
\end{proof}

\begin{cor} \label{c:coh} Let $\mathcal{C}$ be an additive category, and let $\mathcal{D}$ and $\mathcal{E}$ be exact
structures on $\mathcal{C}$. Assume that $\mathcal{C}$ has enough $\mathcal{D}$-injectives and enough
$\mathcal{D}$-projectives.
\begin{enumerate}
\item Assume that $\mathcal{E}$ is $\mathcal{D}$-projectively generated by a class $\mathcal{M}$ of objects of
$\mathcal{C}$. Then the following are equivalent:
\begin{enumerate}[(i)]
\item The class of $\mathcal{D}$-$\mathcal{E}$-divisible objects of $\mathcal{C}$ is closed under
$\mathcal{D}$-deflations.
\item For every $\mathcal{D}$-conflation $Z\rightarrowtail U\twoheadrightarrow V$ with $V\in \mathcal{M}$ and $U$
$\mathcal{D}$-projective, $Z$ is $\mathcal{D}$-projective.
\end{enumerate}
\item Assume that $\mathcal{E}$ is $\mathcal{D}$-injectively generated by a class $\mathcal{M}$ of objects of
$\mathcal{C}$. Then the following are equivalent:
\begin{enumerate}[(i)]
\item The class of $\mathcal{D}$-$\mathcal{E}$-flat objects of $\mathcal{C}$ is closed under $\mathcal{D}$-inflations.
\item For every $\mathcal{D}$-conflation $X\rightarrowtail Y'\twoheadrightarrow Z'$ with $X\in \mathcal{M}$
and $Y'$ $\mathcal{D}$-injective, $Z'$ is $\mathcal{D}$-injective.
\end{enumerate}
\end{enumerate}
\end{cor}

\begin{proof} Use Theorem \ref{t:cohcot} for the $\mathcal{D}$-cotorsion pair $(\mathcal{A},\mathcal{B})$, where
$\mathcal{A}$ is the class of $\mathcal{D}$-projective objects and $\mathcal{B}=\mathcal{C}$ for (1),
and $\mathcal{A}=\mathcal{C}$ and $\mathcal{B}$ is the class of $\mathcal{D}$-injective objects for (2). 
\end{proof}

Let $\mathcal{C}$ be an additive category and $\mathcal{D}$ an exact structure on $\mathcal{C}$. Generalizing the
corresponding notions from module categories, a class $\mathcal{M}$ of objects of $\mathcal{C}$ is called
\emph{$\mathcal{D}$-resolving} (respectively \emph{$\mathcal{D}$-coresolving}) if it contains all
$\mathcal{D}$-projective objects ($\mathcal{D}$-injective objects), it is closed under $\mathcal{D}$-extensions and for
every $\mathcal{D}$-conflation $X\rightarrowtail Y\twoheadrightarrow Z$ with $Y,Z\in \mathcal{M}$ ($X,Y\in
\mathcal{M}$) one has $X\in \mathcal{M}$ ($Z\in \mathcal{M}$).  

\begin{cor} Let $\mathcal{C}$ be an additive category, $\mathcal{D}$ an exact structures on $\mathcal{C}$, and
$(\mathcal{A},\mathcal{B})$ a $\mathcal{D}$-cotorsion pair. Assume that $\mathcal{C}$ has enough
$\mathcal{D}$-injectives and enough $\mathcal{D}$-projectives. View $\mathcal{B}$ as the class of 
relatively divisible objects of $\mathcal{C}$ for some exact structure $\mathcal{D}$-projectively generated 
by a class of objects $\mathcal{N}$, and $\mathcal{A}$ as the class of relatively flat objects of $\mathcal{C}$ 
for some exact structure $\mathcal{D}$-injectively generated by a class of objects $\mathcal{M}$.
Then the following are equivalent:
\begin{enumerate}[(i)]
\item $\mathcal{B}$ is $\mathcal{D}$-coresolving.
\item For every $\mathcal{D}$-conflation $Z\rightarrowtail U\twoheadrightarrow V$ in $\mathcal{C}$ with $V\in
\mathcal{N}$ and $U$ $\mathcal{D}$-projective, $Z\in \mathcal{A}$.
\item $\mathcal{A}$ is $\mathcal{D}$-resolving.
\item For every $\mathcal{D}$-conflation $X\rightarrowtail Y'\twoheadrightarrow Z'$ in $\mathcal{C}$ with $X\in
\mathcal{M}$ and $Y'$ $\mathcal{D}$-injective, $Z'\in \mathcal{B}$.
\end{enumerate}
\end{cor}

\begin{proof} Note that $\mathcal{A}$ contains all $\mathcal{D}$-projective objects, $\mathcal{B}$ contains all
$\mathcal{D}$-injective objects, and both $\mathcal{A}$ and $\mathcal{B}$ are closed under $\mathcal{D}$-extensions.

The implications (i)$\Leftrightarrow$(ii) and (iii)$\Leftrightarrow$(iv) follow by Theorem \ref{t:cohcot}.

(i)$\Rightarrow$(iii) Assume that $\mathcal{B}$ is $\mathcal{D}$-coresolving. 
Let $Z\stackrel{j}\rightarrowtail U\stackrel{p}\twoheadrightarrow V$ with $U,V\in
\mathcal{A}$. Let $X\stackrel{i}\rightarrowtail Y\stackrel{d}\twoheadrightarrow Z$ be a $\mathcal{D}$-conflation with
$X\in \mathcal{B}$. Consider a $\mathcal{D}$-inflation $f:Y\to Y'$ for some $\mathcal{D}$-injective object $Y'$ of
$\mathcal{C}$, and denote by $i'=fi:X\to Y'$ the composed $\mathcal{D}$-inflation and by $d':Y'\to Z'$ its cokernel.
Then there is a morphism $g:Z\to Z'$ such that $gd=d'f$. Now consider the pushout of the $\mathcal{D}$-inflation $j$ and
the morphism $g$ to obtain the commutative diagram from Proposition \ref{p:diagram}. Since $X,Y'\in \mathcal{B}$, we
have $Z'\in \mathcal{B}$, and so ${\rm Ext}^1_{\mathcal{D}}(V,Z')=0$. Since $U\in \mathcal{A}$, it follows that
$Z\in \mathcal{A}$ by Proposition \ref{p:diagram}. Hence $\mathcal{A}$ is $\mathcal{D}$-resolving.

(iii)$\Rightarrow$(i) This is dual to (i)$\Rightarrow$(iii).
\end{proof}

In what follows we establish some further characterizations of relatively divisible and relatively flat objects in exact
categories.

\begin{prop} \label{p:ses} Let $\mathcal{C}$ be an additive category, and let $\mathcal{D}$ and $\mathcal{E}$ be exact
structures on $\mathcal{C}$.
\begin{enumerate}
\item Assume that $\mathcal{C}$ has enough $\mathcal{D}$-injectives. Then the following are equivalent for an object $X$
of $\mathcal{C}$:
\begin{enumerate}[(i)]
\item $X$ is $\mathcal{D}$-$\mathcal{E}$-divisible.
\item There exists an $\mathcal{E}$-conflation $X\rightarrowtail Y\twoheadrightarrow Z$, where $Y$ is
$\mathcal{D}$-injective.
\item There exists an $\mathcal{E}$-conflation $X\rightarrowtail Y\twoheadrightarrow Z$, where $Y$ is
$\mathcal{D}$-$\mathcal{E}$-divisible.
\end{enumerate} 
If $\mathcal{E}$ is $\mathcal{D}$-injectively generated by a class of objects $\mathcal{M}$, then they are also
equivalent to:
\begin{enumerate}
\item[(iv)] For every object $M\in \mathcal{M}$, every morphism $X\to M$ factors through a $\mathcal{D}$-injective
object.
\end{enumerate}
\item Assume that $\mathcal{C}$ has enough $\mathcal{D}$-projectives. Then the following are equivalent for an object
$Z$ of $\mathcal{C}$:
\begin{enumerate}[(i)]
\item $Z$ is $\mathcal{D}$-$\mathcal{E}$-flat.
\item There exists an $\mathcal{E}$-conflation $X\rightarrowtail Y\twoheadrightarrow Z$, where $Y$ is
$\mathcal{D}$-projective.
\item There exists an $\mathcal{E}$-conflation $X\rightarrowtail Y\twoheadrightarrow Z$, where $Y$ is
$\mathcal{D}$-$\mathcal{E}$-flat. 
\end{enumerate} 
If $\mathcal{E}$ is $\mathcal{D}$-projectively generated by a class of objects $\mathcal{M}$, then they are also
equivalent to:
\begin{enumerate}
\item[(iv)] For every object $M\in \mathcal{M}$, every morphism $M\to Z$ factors through a $\mathcal{D}$-projective
object.
\end{enumerate}
\end{enumerate}
\end{prop}

\begin{proof} (1) (i)$\Rightarrow$(ii) Assume that $X$ is $\mathcal{D}$-$\mathcal{E}$-divisible. Let $X\rightarrowtail 
Y\twoheadrightarrow M$ be a $\mathcal{D}$-conflation with $Y$ $\mathcal{D}$-injective. By hypothesis, it must be an
$\mathcal{E}$-conflation.

(ii)$\Rightarrow$(iii) This is clear.

(iii)$\Rightarrow$(i) Assume that there exists an $\mathcal{E}$-conflation $X\stackrel{\beta}\rightarrowtail
Y\twoheadrightarrow Z$, where $Y$ is $\mathcal{D}$-$\mathcal{E}$-divisible. Let $\alpha:X\rightarrowtail X'$ be a
$\mathcal{D}$-inflation. Now consider the pushout of $\alpha$ and $\beta$ to obtain the following commutative diagram:
\[\SelectTips{cm}{}
\xymatrix{
X \ar@{>->}[r]^-{\beta} \ar@{>->}[d]_{\alpha} & Y \ar@{->>}[r] \ar@{>->}[d]^{\gamma} & Z \ar@{=}[d] \\ 
X' \ar@{>->}[r]^-{\delta} & Y' \ar@{->>}[r] & Z
}\]
where the last row is an $\mathcal{E}$-conflation. Then the $\mathcal{D}$-inflation $\gamma:Y\to Y'$ is an
$\mathcal{E}$-inflation, because $Y$ is $\mathcal{D}$-$\mathcal{E}$-divisible. It follows that
$\delta\alpha=\gamma\beta$ is an $\mathcal{E}$-inflation, hence $\alpha$ must be an $\mathcal{E}$-inflation. Therefore,
$X$ is $\mathcal{D}$-$\mathcal{E}$-divisible.

(i)$\Rightarrow$(iv) Assume that $X$ is $\mathcal{D}$-$\mathcal{E}$-divisible. Let $M\in \mathcal{M}$ and consider a
morphism $f:X\to M$. There is a $\mathcal{D}$-conflation $X\rightarrowtail Y\twoheadrightarrow Z$ for some
$\mathcal{D}$-injective object $Y$. This must be an $\mathcal{E}$-conflation, because $X$ is
$\mathcal{D}$-$\mathcal{E}$-divisible. But $\mathcal{E}=\mathcal{E}^{\mathcal{M}}_G$, hence $M$ is injective with
respect to $\mathcal{E}$-conflations. Then $f$ factors through the $\mathcal{D}$-injective object $Y$.

(iv)$\Rightarrow$(ii) Assume that (iv) holds. There is a $\mathcal{D}$-conflation $X\stackrel{\alpha}\rightarrowtail
Y\twoheadrightarrow Z$ for some $\mathcal{D}$-injective object $Y$. Let $f:X\to M$ be a morphism in $\mathcal{C}$ with
$M\in \mathcal{M}$. By hypothesis, $f$ factors through a $\mathcal{D}$-injective object $Y'$, hence there exist
morphisms $g:Y'\to M$ and $h:X\to Y'$ such that $f=gh$. By the $\mathcal{D}$-injectivity of $Y'$, there exists a
morphism $\beta:Y\to Y'$ such that $\beta\alpha=h$. Then $g\beta\alpha=f$, hence $M$ is injective with respect to the
$\mathcal{D}$-conflation. But $\mathcal{E}=\mathcal{E}^{\mathcal{M}}_G$, hence the initial $\mathcal{D}$-conflation
must be an $\mathcal{E}$-conflation, which concludes the proof. 
\end{proof}

Relative divisibility and relative flatness may be further generalized with respect to some class of objects. These new
notions are motivated by some subsequent applications to finitely accessible additive categories. 

\begin{defi} \rm Let $\mathcal{C}$ be an additive category, let $\mathcal{D}$ and $\mathcal{E}$ be exact structures on
$\mathcal{C}$ and $\mathcal{A}$ a class of objects in $\mathcal{C}$. An object $X$ of $\mathcal{C}$ is called:
\begin{enumerate}
\item \emph{$\mathcal{D}$-$\mathcal{E}$-$\mathcal{A}$-divisible} if every $\mathcal{D}$-inflation $X\rightarrowtail A$
with $A\in \mathcal{A}$ is an $\mathcal{E}$-inflation.
\item \emph{$\mathcal{D}$-$\mathcal{E}$-$\mathcal{A}$-flat} if every $\mathcal{D}$-deflation $A\twoheadrightarrow X$
with $A\in \mathcal{A}$ is an $\mathcal{E}$-deflation.
\end{enumerate}
\end{defi}

Some closure properties of relatively divisible and relatively flat objects in exact categories from Proposition
\ref{p:closureDE} can be generalized with respect to a class of objects as follows.

\begin{prop} \label{p:closure} Let $\mathcal{C}$ be an additive category, let $\mathcal{D}$ and $\mathcal{E}$ be exact
structures on $\mathcal{C}$, and $\mathcal{A}$ a class of objects in $\mathcal{C}$.  
\begin{enumerate} \item 
\begin{enumerate}[(i)] 
\item Assume that $\mathcal{A}$ is closed under $\mathcal{E}$-deflations. Then the class of
$\mathcal{D}$-$\mathcal{E}$-$\mathcal{A}$-divisible objects of $\mathcal{C}$ is closed under $\mathcal{D}$-extensions. 
\item Assume that $\mathcal{C}$ has products and $\mathcal{E}$ is closed under products. Then the class of
$\mathcal{D}$-$\mathcal{E}$-$\mathcal{A}$-divisible objects of $\mathcal{C}$ is closed under products.
\item Assume that $\mathcal{C}$ has direct limits and $\mathcal{E}$ is closed under direct limits. Let $(X_i,f_{ij})_I$
be a direct system of $\mathcal{D}$-$\mathcal{E}$-$\mathcal{A}$-divisible objects of $\mathcal{C}$ with direct limit
$(\underset{\longrightarrow}\lim \,X_i,f_i)$ such that each $f_{ij}$ is a $\mathcal{D}$-inflation. Then
$\underset{\longrightarrow}\lim \,X_i$ is $\mathcal{D}$-$\mathcal{E}$-$\mathcal{A}$-divisible. In particular, the class
of $\mathcal{D}$-$\mathcal{E}$-$\mathcal{A}$-divisible objects of $\mathcal{C}$ is closed under coproducts.
\end{enumerate}
\item 
\begin{enumerate}[(i)]
\item Assume that $\mathcal{A}$ is closed under $\mathcal{E}$-inflations. Then the class of
$\mathcal{D}$-$\mathcal{E}$-$\mathcal{A}$-flat objects of $\mathcal{C}$ is closed under $\mathcal{D}$-extensions. 
\item Assume that $\mathcal{C}$ has coproducts and $\mathcal{E}$ is closed under coproducts. Then the class of
$\mathcal{D}$-$\mathcal{E}$-$\mathcal{A}$-flat objects of $\mathcal{C}$ is closed under coproducts.
\end{enumerate}
\end{enumerate}
\end{prop}

\begin{proof} (1) (i) Let $X\stackrel{f}\rightarrowtail Y\stackrel{g}\twoheadrightarrow Z$ be a
$\mathcal{D}$-conflation in
$\mathcal{C}$ such that $X$ and $Z$ are $\mathcal{D}$-$\mathcal{E}$-$\mathcal{A}$-divisible. Let $j:Y\rightarrowtail A$
be a $\mathcal{D}$-inflation with $A\in \mathcal{A}$ and consider the pushout of $g$ and $j$. Then we obtain the
following commutative diagram:
\[\SelectTips{cm}{}
\xymatrix{
X \ar@{>->}[r]^-{f} \ar@{=}[d] & Y \ar@{->>}[r]^-{g} \ar@{>->}[d]_{j} & Z \ar@{>->}[d]^{j'} \\ 
X \ar@{>->}[r]_-{f'} & A \ar@{->>}[r]_-{g'} \ar@{->>}[d]_{p} & A' \ar@{->>}[d]^{p'} \\
& C \ar@{=}[r] & C 
}\] 
where the rows and columns are $\mathcal{D}$-conflations. Since $X$ is
$\mathcal{D}$-$\mathcal{E}$-$\mathcal{A}$-divisible, $g'$ is an $\mathcal{E}$-deflation. Then we have $A'\in
\mathcal{A}$, because $\mathcal{A}$ is closed under $\mathcal{E}$-deflations by hypothesis. Since $Z$ is
$\mathcal{D}$-$\mathcal{E}$-$\mathcal{A}$-divisible, it follows that $p'$ is an $\mathcal{E}$-deflation. Then $p=p'g'$
is an $\mathcal{E}$-deflation, and so $Y$ is $\mathcal{D}$-$\mathcal{E}$-$\mathcal{A}$-divisible.

(ii) Let $(X_i)_{i\in I}$ be a family of $\mathcal{D}$-$\mathcal{E}$-$\mathcal{A}$-divisible objects of
$\mathcal{C}$ and $X=\prod_{i\in I}X_i$. Let $f:X\rightarrowtail A$ be a $\mathcal{D}$-inflation with $A\in
\mathcal{A}$. For every $k\in I$ denote $X'_k=\prod_{i\in I, i\neq k}X_i$ and consider the canonical split monomorphism
$\alpha_k:X_k\to X$. Then $\alpha'_k=f\alpha_k$ is a $\mathcal{D}$-inflation and we obtain a pullback-pushout
commutative diagram of the form:
\[\SelectTips{cm}{}
\xymatrix{
X_k \ar@{>->}[d]_{\alpha_k} \ar@{=}[r] & X_k \ar@{>->}[d]^{\alpha'_k} & \\
X \ar@{>->}[r]^-{f} \ar@{->>}[d]_{\beta_k} & A \ar@{->>}[r]^-{g} \ar@{->>}[d]^{\beta'_k} & Z \ar@{=}[d] \\ 
X'_k \ar@{>->}[r]_-{f'_k} & Y_k \ar@{->>}[r]_-{g'_k} & Z
}\]
Since $X_k$ is $\mathcal{D}$-$\mathcal{E}$-$\mathcal{A}$-divisible, it follows that the second column is an
$\mathcal{E}$-conflation. Then the induced morphism $\prod_{k\in I}\alpha'_k:\prod_{k\in I}X_k\to A^I$ is an
$\mathcal{E}$-inflation. But $\prod_{k\in I}\alpha'_k$ is the composition of morphisms $X=\prod_{k\in
I}X_k\stackrel{f}\to A\stackrel{j}\to A^I$, where $j$ is the canonical monomorphism. Then $f:X\rightarrowtail A$ is an
$\mathcal{E}$-inflation. Therefore, $X$ is $\mathcal{D}$-$\mathcal{E}$-$\mathcal{A}$-divisible.

(iii) Denote $X=\underset{\longrightarrow}\lim \,X_i$. Note that, since each $f_{ij}$ is a $\mathcal{D}$-inflation, so
is each canonical morphism $f_i:X_i\to X$. Let $X\stackrel{u}\rightarrowtail A\stackrel{v}\twoheadrightarrow Z$ be a
$\mathcal{D}$-conflation in $\mathcal{C}$ with $A\in \mathcal{A}$. For every $i\in I$, $u_i=uf_i$ is a
$\mathcal{D}$-inflation and we have a commutative diagram 
\[\SelectTips{cm}{}
\xymatrix{
X_i \ar@{>->}[d]_{f_i} \ar@{>->}[r]^{u_i} & A \ar@{=}[d] \ar@{->>}[r]^{v_i} & Z_i \ar[d] \\
X \ar@{>->}[r]^-{u} & A \ar@{->>}[r]^-{v} & Z \\ 
}\]
Since $X_i$ is $\mathcal{D}$-$\mathcal{E}$-$\mathcal{A}$-divisible, the upper row is an $\mathcal{E}$-conflation.
Now take the direct limit of the $\mathcal{E}$-conflations $X_i\stackrel{u_i}\rightarrowtail
A\stackrel{v_i}\twoheadrightarrow Z$ for $i\in I$ and use the closure of $\mathcal{E}$ under direct limits in order to
deduce that $X\stackrel{u}\rightarrowtail A\stackrel{v}\twoheadrightarrow Z$ is an $\mathcal{E}$-conflation. This shows
that $X=\underset{\longrightarrow}\lim \,X_i$ is $\mathcal{D}$-$\mathcal{E}$-$\mathcal{A}$-divisible.
\end{proof}

\begin{theor} \label{t:covenv} Let $\mathcal{C}$ be an additive category, $\mathcal{D}$ an exact structure on
$\mathcal{C}$, and $(\mathcal{A},\mathcal{B})$ a $\mathcal{D}$-perfect $\mathcal{D}$-cotorsion pair in $\mathcal{C}$. 
\begin{enumerate}
\item Consider an exact structure $\mathcal{E}$ on $\mathcal{C}$ such that $\mathcal{A}$ coincides with the class of
$\mathcal{D}$-$\mathcal{E}$-flat objects of $\mathcal{C}$ and assume that $\mathcal{E}\subseteq \mathcal{D}$. Then the
following are equivalent for an object $X$ of $\mathcal{C}$:
\begin{enumerate}[(i)]
\item $X$ is $\mathcal{D}$-$\mathcal{E}$-$\mathcal{A}$-divisible.
\item The $\mathcal{B}$-envelope of $X$ relative to $\mathcal{D}$ is
$\mathcal{D}$-$\mathcal{E}$-$\mathcal{A}$-divisible.
\item There exists an $\mathcal{E}$-conflation $X\rightarrowtail B\twoheadrightarrow A$, where $A\in \mathcal{A}$
and $B\in \mathcal{B}$ is $\mathcal{D}$-$\mathcal{E}$-$\mathcal{A}$-divisible.
\item There exists an $\mathcal{E}$-conflation $X\rightarrowtail Y\twoheadrightarrow A$, where $A\in \mathcal{A}$ and
$Y$ is $\mathcal{D}$-$\mathcal{E}$-$\mathcal{A}$-divisible.
\end{enumerate}
\item Consider an exact structure $\mathcal{E}$ on $\mathcal{C}$ such that $\mathcal{B}$ coincides with the class of
$\mathcal{D}$-$\mathcal{E}$-divisible objects of $\mathcal{C}$ and assume that $\mathcal{E}\subseteq \mathcal{D}$. Then
the following are equivalent for an object $Z$ of $\mathcal{C}$:
\begin{enumerate}[(i)]
\item $Z$ is $\mathcal{D}$-$\mathcal{E}$-$\mathcal{B}$-flat.
\item The $\mathcal{A}$-cover of $Z$ relative to $\mathcal{D}$ is $\mathcal{D}$-$\mathcal{E}$-$\mathcal{B}$-flat.
\item There exists an $\mathcal{E}$-conflation $B\rightarrowtail A\twoheadrightarrow Z$, where $A\in \mathcal{A}$ is
$\mathcal{D}$-$\mathcal{E}$-$\mathcal{B}$-flat and $B\in \mathcal{B}$.
\item There exists an $\mathcal{E}$-conflation $B\rightarrowtail Y\twoheadrightarrow Z$, where $Y$ is
$\mathcal{D}$-$\mathcal{E}$-$\mathcal{B}$-flat and $B\in \mathcal{B}$.
\end{enumerate}
\end{enumerate}
\end{theor}

\begin{proof} (1) The existence of the required exact structure $\mathcal{E}$ on $\mathcal{C}$ follows by Proposition 
\ref{p:identif}.

(i)$\Rightarrow$(ii) Assume that $X$ is $\mathcal{D}$-$\mathcal{E}$-$\mathcal{A}$-divisible. Consider
the $\mathcal{B}$-envelope $X\rightarrowtail B(X)$ of $X$ relative to $\mathcal{D}$, which does exist and is a
$\mathcal{D}$-inflation, because the $\mathcal{D}$-cotorsion pair $(\mathcal{A},\mathcal{B})$ is $\mathcal{D}$-perfect.
Let $B(X)\rightarrowtail A$ be a $\mathcal{D}$-inflation with $A\in \mathcal{A}$. Consider the $\mathcal{B}$-envelope
$A\rightarrowtail B(A)$ of $A$ relative to $\mathcal{D}$ and the induced composition $X\rightarrowtail
B(X)\rightarrowtail A\rightarrowtail B(A)$ of morphisms, which we denote by $\alpha$. Then $\alpha:X\to B(A)$ is a
$\mathcal{D}$-inflation, as a composition of $\mathcal{D}$-inflations. Consider the induced $\mathcal{D}$-conflations
$A\rightarrowtail B(A)\twoheadrightarrow A'$ and $X\rightarrowtail
B(A)\twoheadrightarrow Z$. Since $A,A'\in \mathcal{A}$ and $\mathcal{A}$ is closed under $\mathcal{D}$-extensions, we
have $B(A)\in \mathcal{A}$. Since $X$ is $\mathcal{D}$-$\mathcal{E}$-$\mathcal{A}$-divisible, it follows that
$\alpha:X\to B(A)$ is an $\mathcal{E}$-inflation. Then $Z\in \mathcal{A}$, because $\mathcal{A}$ is closed under
$\mathcal{E}$-deflations by Propositions \ref{p:closureDE} and \ref{p:identif}. We have the following induced
commutative diagram:
\[\SelectTips{cm}{}
\xymatrix{
X \ar@{>->}[r]^-{\beta} \ar@{=}[d] & B(X) \ar@{-->}[d]^{\gamma} & \\ 
X \ar@{>->}[r]^-{\alpha} \ar@{>->}[d]_-{\beta} & B(A) \ar@{>->}[d] \ar@{->>}[r] \ar@{-->}[dl]^{\delta} & Z
\ar@{=}[d] \\
B(X) \ar@{>->}[r] & Y \ar@{->>}[r] & Z 
}\] 
where the lower $\mathcal{D}$-conflation is obtained by taking the pushout of $\alpha$ and $\beta$. The morphism
$\gamma:B(X)\to B(A)$ such that $\gamma\beta=\alpha$ does exist because $B(X)$ is a $\mathcal{B}$-preenvelope of $X$
(relative to $\mathcal{D}$).
Since $Z\in \mathcal{A}$ and $B(X)\in \mathcal{B}$, it follows that ${\rm Ext}^1_{\mathcal{D}}(Z,B(X))=0$, hence the
lower $\mathcal{D}$-conflation splits. By the Homotopy Lemma \cite[7.16]{Wisb} there exists the required morphism
$\delta:B(A)\to B(X)$ such that $\delta\alpha=\beta$. Then $\delta\gamma\beta=\beta$, whence it follows that
$\delta\gamma$ is an automorphism, using the $\mathcal{B}$-envelope (relative to $\mathcal{D}$) condition for $B(X)$.
Hence $\gamma:B(X)\to B(A)$ is a split monomorphism, which implies that $B(X)\rightarrowtail A$ is a split monomorphism,
and so it is an $\mathcal{E}$-inflation. Therefore, $B(X)$ is $\mathcal{D}$-$\mathcal{E}$-$\mathcal{A}$-divisible.

(ii)$\Rightarrow$(iii) By hypothesis, there is a $\mathcal{D}$-conflation $X\rightarrowtail B(X)\twoheadrightarrow A$,
where $A\in \mathcal{A}$ and $B(X)$ is $\mathcal{D}$-$\mathcal{E}$-$\mathcal{A}$-divisible. Since every object of
$\mathcal{A}$ is $\mathcal{D}$-$\mathcal{E}$-flat, the $\mathcal{D}$-conflation is also an $\mathcal{E}$-conflation. 

(iii)$\Rightarrow$(iv) This is clear.

(iv)$\Rightarrow$(i) Assume that there exists an $\mathcal{E}$-conflation $X\stackrel{\beta}\rightarrowtail
Y\twoheadrightarrow A$, where $A\in \mathcal{A}$ and $Y$ is $\mathcal{D}$-$\mathcal{E}$-$\mathcal{A}$-divisible. Let
$\alpha:X\rightarrowtail A'$ be a $\mathcal{D}$-inflation with $A'\in \mathcal{A}$. Now consider the pushout of $\alpha$
and $\beta$ to obtain the following commutative diagram:
\[\SelectTips{cm}{}
\xymatrix{
X \ar@{>->}[r]^-{\beta} \ar@{>->}[d]_{\alpha} & Y \ar@{->>}[r] \ar@{>->}[d]^{\gamma} & A \ar@{=}[d] \\ 
A' \ar@{>->}[r]^-{\delta} & Y' \ar@{->>}[r] & A 
}\]
where the last row is an $\mathcal{E}$-conflation. Since $\mathcal{E}\subseteq \mathcal{D}$ and $\mathcal{A}$ is closed
under $\mathcal{D}$-extensions, we have $Y'\in \mathcal{A}$. Then the $\mathcal{D}$-inflation $\gamma:Y\to Y'$ is an
$\mathcal{E}$-inflation, because $Y$ is $\mathcal{D}$-$\mathcal{E}$-$\mathcal{A}$-divisible. It follows that
$\delta\alpha=\gamma\beta$ is an $\mathcal{E}$-inflation, hence $\alpha$ must be an $\mathcal{E}$-inflation. Therefore,
$X$ is $\mathcal{D}$-$\mathcal{E}$-$\mathcal{A}$-divisible. 
\end{proof}

\end{document}